\documentclass[12pt,leqno]{article}
\usepackage{amsfonts}
\usepackage{amsmath, amsthm, amsfonts, amssymb, color}
\usepackage{mathrsfs}
\usepackage{url}
\usepackage{color}
\newtheorem{thm}{Theorem}[section]

\newtheorem{lem}[thm]{Lemma}

\theoremstyle{definition}

\newcommand{\scr}[1]{\mathscr #1}
\definecolor{wco}{rgb}{0.5,0.2,0.3}

\numberwithin{equation}{section} \theoremstyle{remark}

\newcommand{\ua}{\uparrow}

\title{{\bf      Wasserstein  Convergence  for Empirical Measures of Subordinated  Diffusions on Riemannian Manifolds }\footnote{Supported in
 part by  NNSFC (11771326, 11831014, 11921001).} }
\author{
{\bf    Feng-Yu Wang$^{a),b)}$ \, \  Bingyao Wu$^{a)}$  }\\
\footnotesize{$^{a)}$ Center for Applied Mathematics, Tianjin University, Tianjin 300072, China }\\
 \footnotesize{ $^{b)}$ Department of Mathematics,
Swansea University, Bay Campus, Swansea, SA1 8EN, United Kingdom}\\
 \footnotesize{  wangfy@tju.edu.cn, F.-Y.Wang@swansea.ac.uk; bingyaowu@163.com } }

\begin{document}
\allowdisplaybreaks
\def\R{\mathbb R}  \def\ff{\frac} \def\ss{\sqrt} \def\B{\mathbf
B}\def\TO{\mathbb T}
\def\I{\mathbb I_{\pp M}}\def\p<{\preceq}
\def\N{\mathbb N} \def\kk{\kappa} \def\m{{\bf m}}
\def\ee{\varepsilon}\def\ddd{D^*}
\def\dd{\delta} \def\DD{\Delta} \def\vv{\varepsilon} \def\rr{\rho}
\def\<{\langle} \def\>{\rangle} \def\GG{\Gamma} \def\gg{\gamma}
  \def\nn{\nabla} \def\pp{\partial} \def\E{\mathbb E}
\def\d{\text{\rm{d}}} \def\bb{\beta} \def\aa{\alpha} \def\D{\scr D}
  \def\si{\sigma} \def\ess{\text{\rm{ess}}}
\def\beg{\begin} \def\beq{\begin{equation}}  \def\F{\scr F}
\def\Ric{{\rm Ric}} \def\Hess{\text{\rm{Hess}}}
\def\e{\text{\rm{e}}} \def\ua{\underline a} \def\OO{\Omega}  \def\oo{\omega}
 \def\tt{\tilde}
\def\cut{\text{\rm{cut}}} \def\P{\mathbb P} \def\ifn{I_n(f^{\bigotimes n})}
\def\C{\scr C}      \def\aaa{\mathbf{r}}     \def\r{r}
\def\gap{\text{\rm{gap}}} \def\prr{\pi_{{\bf m},\varrho}}  \def\r{\mathbf r}
\def\Z{\mathbb Z} \def\vrr{\varrho} \def\ll{\lambda}
\def\L{\scr L}\def\Tt{\tt} \def\TT{\tt}\def\II{\mathbb I}
\def\i{{\rm in}}\def\Sect{{\rm Sect}}  \def\H{\mathbb H}
\def\M{\scr M}\def\Q{\mathbb Q} \def\texto{\text{o}} \def\LL{\Lambda}
\def\Rank{{\rm Rank}} \def\B{\scr B} \def\i{{\rm i}} \def\HR{\hat{\R}^d}
\def\to{\rightarrow}\def\l{\ell}\def\iint{\int}
\def\EE{\scr E}\def\Cut{{\rm Cut}}\def\W{\mathbb W}
\def\A{\scr A} \def\Lip{{\rm Lip}}\def\S{\mathbb S}
\def\BB{\mathbb B}\def\Ent{{\rm Ent}} \def\i{{\rm i}}\def\itparallel{{\it\parallel}}
\def\g{{\mathbf g}}\def\Sect{{\mathcal Sec}}\def\T{\mathcal T}\def\V{{\bf V}}
\def\PP{{\bf P}}\def\HL{{\bf L}}\def\Id{{\rm Id}}\def\f{{\bf f}}\def\cut{{\rm cut}}
\def\Ss{\mathbb S}
\def\BL{\scr A}\def\Pp{\mathbb P}\def\Pp{\mathbb P} \def\Ee{\mathbb E}

\maketitle

\begin{abstract} Let $M$ be a   connected compact Riemannian manifold  possibly with a  boundary $\pp M$, let $V\in C^2(M)$ such that $\mu(\d x):=\e^{V(x)}\d x$ is a probability measure, where $\d x$ is the volume measure, and let
$L=\DD+\nn V$.  As a continuation to  \cite{eWZ} where  convergence  in the quadratic Wasserstein distance  $\W_2$ is atudied for the empirical measures of the $L$-diffusion process (with reflecting boundary if $\pp M\ne\emptyset$),
   this paper   presents  the exact convergence rate for   the subordinated process. In particular,
letting $(\mu_t^{\aa})_{t>0}$ ($\aa\in (0,1))$ be the empirical measures  of  the Markov process generated by
$L^\aa:= -(-L)^\aa$,  when $\pp M$ is  empty or convex we have
$$ \lim_{t\to \infty}  \big\{t \E^x [\W_2(\mu_{t}^\aa,\mu)^2]\big\}= \sum_{i=1}^\infty\ff 2 {\ll_i^{1+\aa}}\ \text{ uniformly\ in\ }  x\in M,$$ where  $\E^x$ is the expectation   for the   process starting at point $x$, $\{\ll_i\}_{i\ge 1}$ are non-trivial  (Neumann) eigenvalues of $-L$.  In general,
 $$\E^x [\W_2(\mu_{t}^\aa,\mu)^2]  \beg{cases} \asymp t^{-1}, &\text{if}\ d<2(1+\aa),\\
 \asymp  t^{-\ff 2 {d-2\aa}}, &\text{if} \ d>2(1+\aa),\\
 \preceq t^{-1}\log (1+t), &\text{if} \ d=2(1+\aa), \text{i.e.}\  \aa=\ff 1 2, d=3\end{cases}$$ holds uniformly in $x\in M$, where in the last case
 $\E^x [\W_1(\mu_{t}^\aa,\mu)^2]\succeq  t^{-1}\log (1+t)$    holds for $M=\TO^3$ and $V=0.$
 \end{abstract} \noindent
 \vskip 2cm

\section{Introduction  }

Recently,        sharp convergence rate  in the Wasserstein distance has been derived in \cite{eWZ} for empirical measures of symmetric diffusion processes on compact Riemannian manifolds,   see   \cite{eW1, eW2, eW3, eW4} for further  study of Dirichlet diffusion processes and SDEs/SPDEs, and see
\cite{eAST, BLG14, [25]} and references within for  earlier  results on    i.i.d. random variables and discrete time Markov chains.
 In this paper,
we aim to extend the main results of  \cite{eWZ} to jump processes, for which a natural model is  the subordination  of   diffusion processes.

Let $M$ be a $d$-dimensional    connected  compact Riemannian manifold   possibly with a smooth  boundary $\pp M$.  Let $V\in C^2(M)$ such that $\mu(\d x)=\e^{V(x)} \d x$ is a probability measure on $M$, where $\d x$ is the Riemannian volume measure on $M$. Then  the   (reflecting, if $\pp M\ne \emptyset$) diffusion process $X_t$ generated by $L:=\DD+\nn V$ on $M$ is reversible; i.e. the associated diffusion semigroup $\{ P_t\}_{t\ge 0}$ is symmetric in $L^2(\mu)$, where
 $$ P_tf(x):= \E^x f(X_t),\ \ t\ge 0, f\in \B_b(M). $$
  Here, $\E^x$ is  the expectation  taken for the diffusion process $\{X_t\}_{t\ge 0} $ with
 $X_0=x,$ and we will use $\P^x$ to denote the associated probability measure.
 In general, for  $\nu\in \scr P$ (the set of all probability measures on $M$), let $\E^\nu$ and $\P^\nu$ be the expectation and probability taken for the diffusion process with initial distribution $\nu$.  For any $\nu\in \scr P$ and $t\ge 0$, $\nu P_t:= \P^\nu( X_t \in \cdot)$ is the distribution of $X_t$ with initial distribution $\nu$.

 A function $B\in C^\infty((0,\infty);[0,\infty))\cap C([0,\infty); [0,\infty))$ is called a Bernstein function if
 $$ (-1)^{n-1} \ff{\d^n}{\d r^n}B(r)\ge0,\ \ n\in\mathbb N, r>0.$$
We will use the following classes of  Bernstein functions:
\beg{align*} & {\bf B}:= \big\{B: B\text{\ is\ a\ Bernstein\ function\ with } B(0)=0, B'(0)>0\big\},\\
& \BB:= \bigg\{B\in {\bf B}:\  \int_1^\infty r^{\ff d 2 -1}\e^{-t B(r)}\d r<\infty\ \text{for}\ t>0\bigg\}.\end{align*}
  For each $B\in {\bf  B}$, there exists a unique stable process $S_t^B$ on $[0,\infty)$ with Laplace transform
  \beq\label{LT} \E \e^{-\ll S_t^B}= \e^{-t B(\ll)},\ \ t,\ll \ge 0.\end{equation}
 Moreover, for any $\aa\in [0,1]$, let
 $$\BB^\aa:=\Big\{B\in \BB:\ \liminf_{\ll\to\infty}  \ll^{-\aa} B(\ll)>0\Big\},\ \ \BB_\aa:= \Big\{B\in \BB:\ \limsup_{\ll\to\infty} \ll^{-\aa} B(\ll)<\infty\Big\}.$$

 For any $B\in {\bf B}$, let $X_t^B$ be the Markov process on $M$ generated by $B(L):=-B(-L)$, which can be constructed as the time change (subordination) of  $X_t$:
 $$X_t^B= X_{S^B_t},\ \ t\ge 0,$$
 where $(S^B_t)_{t\ge 0} $ is the stable process satisfying \eqref{LT}  independent of $(X_t)_{t\ge 0}$.
  We consider the empirical measure
 $$\mu_t^B:=\ff 1 t \int_0^t \dd_{X_{s}^B}\d s,\ \ t>0.$$

Let   $\rr$ be the Riemiannian distance   (i.e. the length of shortest curve linking two points) on $M$.
For any $p>0$,  the $L^p$-Wasserstein distance $\W_p$ is defined by
 $$\W_p(\mu_1,\mu_2):= \inf_{\pi\in \C(\mu_1,\mu_2)} \bigg(\int_{M\times M} \rr(x,y)^p \pi(\d x,\d y) \bigg)^{\ff 1 {p\lor 1}},\ \ \mu_1,\mu_2\in \scr P,$$
 where $\C(\mu_1,\mu_2)$ is the set of all probability measures on $M\times M$ with marginal distributions $\mu_1$ and $\mu_2$. A measure $\pi\in \C(\mu_1,\mu_2)$ is called a coupling of $\mu_1$ and $\mu_2$.

Since $M$ is connected and compact, $L$ has  discrete spectrum and all   eigenvalues $\{\ll_i\}_{i\ge 0}$ of $-L$ listed in the increasing order counting multiplicities satisfy (see for instance \cite{Chavel})
 \beq\label{*2} \kk^{-1} i^{\ff 2 d} \le \ll_i\le \kk i^{\ff 2 d},\ \ i\ge 0\end{equation}
for some constant $\kk>1$. Our main results are  stated as  follows, which cover the corresponding assertions derived in \cite{eWZ} for $B(\ll)=\ll$.

\beg{thm}[Lower bound estimates] \label{T1.1} Let $B\in \BB$.
\beg{enumerate} \item[$(1)$] There exists a constant $c\in (0,1]$ with $c=1$ when $\pp M$ is empty or convex, such that
$$ \liminf_{t\to \infty}  \inf_{x\in M}    \Big\{t \E^x [\W_2(\mu^B_{t},\mu)^2]\Big\}\ge   c  \sum_{i=1}^\infty\ff 2 {\ll_i B(\ll_i)}.$$
\item[$(2)$] Let $B\in \BB_\aa$   for some $\aa\in [0,1]$.
If $d>2(1+\aa)$, then  for any $p>0$,
$$  \liminf_{t\to \infty}  \inf_{x\in M}    \Big\{t^{\ff 2{d-2\aa}}\big( \E^x [\W_p(\mu^B_{t},\mu)]\big)^{\ff 2{p\land 1}} \Big\}>0.$$
\item[$(3)$] Let $B(\ll)=\ll^\aa$   for some $\aa\in [0,1]$.   If  $d=2(1+\aa)\ ($i.e. $\aa=1$ and $d=4$, or $\aa=\ff 1 2$ and $d=3)$, $M=\TO^d$ and $V=0$, then
$$  \liminf_{t\to \infty}  \inf_{x\in M}    \Big\{\ff{t}{\log t} \big(\E^x [\W_1(\mu^B_{t},\mu)]\big)^2\Big\}>0.$$
\end{enumerate}
   \end{thm}

\beg{thm}[Upper  bound estimates] \label{T1.2} Let $B\in \BB^\aa$   for some $\aa\in [0,1]$.
\beg{enumerate} \item[$(1)$] If $d<2(1+\aa)$, then
$$ \limsup_{t\to \infty}  \sup_{x\in M}    \Big\{t \E^x [\W_2(\mu^B_{t},\mu)^2]\Big\}\le     \sum_{i=1}^\infty\ff 2 {\ll_i B(\ll_i)}<\infty.$$
\item[$(2)$] If   $d>2(1+\aa)$, then
$$ \limsup_{t\to \infty}  \sup_{x\in M}    \Big\{t^{\ff{2}{d-2\aa}} \E^x [\W_2(\mu^B_{t},\mu)^2]\Big\}<\infty.$$
\item[$(3)$] If   $d=2(1+\aa)$, i.e. either $\aa=1$ and $d=4$, or $\aa=\ff 1 2$ and $d=3$, then
$$ \limsup_{t\to \infty}  \sup_{x\in M}    \Big\{ \ff t{\log t}  \E^x [\W_2(\mu^B_{t},\mu)^2]\Big\}<\infty.$$
\end{enumerate}
   \end{thm}

The following is a straightforward consequence of Theorems \ref{T1.1} and \ref{T1.2}.

\beg{cor}\label{C1.3} Let $B\in  \BB^\aa\cap \BB_\aa$    for  some $\aa\in [0,1]$.
\beg{enumerate} \item[$(1)$] If $\pp M$ is empty or convex, then
\beq\label{UNF}   \lim_{t\to \infty}     \Big\{t \E^x [\W_2(\mu^B_{t},\mu)^2]\Big\}=    \sum_{i=1}^\infty\ff 2 {\ll_iB(\ll_i)}\end{equation}
uniformly in $x\in M$, where the limit is finite if and only if $d<2(1+\aa)$.
In general, there exists a constant $c\in (0,1]$ such that
\beq\label{UEST} \beg{split}  &c\sum_{i=1}^\infty\ff 2 {\ll_i B(\ll_i)} \le \liminf_{t\to \infty}  \inf_{x\in M}    \Big\{t \E^x [\W_2(\mu^B_{t},\mu)^2]\Big\}\\
&\le \limsup_{t\to \infty}  \sup_{x\in M}    \Big\{t \E^x [\W_2(\mu^B_{t},\mu)^2]\Big\}
\le   \sum_{i=1}^\infty\ff 2 {\ll_i B(\ll_i)}.\end{split} \end{equation}
 \item[$(2)$] If $d> 2(1+\aa)$, then for any $\vv\in (0,\aa)$ there exist constants $c>c(\vv)>0$ such that
 \beg{align*} &c(\vv) t^{-\ff{2}{d-2\aa}} \le \inf_{x\in M}\big( \E^x [\W_\vv(\mu_t^B,\mu)]\big)^{\ff 2\vv} \\
 &\le  \inf_{x\in M}  \E^x [\W_2(\mu_t^B,\mu)^2] \le\sup_{x\in M}    \E^x  [\W_2(\mu^B_{t},\mu)^2]\le c t^{-\ff{2}{d-2\aa}},\ \ \ t\ge 1.\end{align*}
 \item[$(3)$]  Let $d= 2(1+\aa), $ i.e. either $d=3$ and $\aa=\ff 1 2$, or $\aa=1$ and $d=4$.   Then there exists a constant $c>0$ such that
$$\sup_{x\in M} \E^x [\W_2(\mu^B_{t},\mu)^2]\le c t^{-1} \log t,\ \ t\ge 2.$$
On the other hand, when $B(\ll)=\ll^\aa, M=\TO^d$ and $V=0$, then there exists a constant $c'>0$ such that
$$\inf_{x\in M} \E^x [\W_1(\mu^B_{t},\mu)^2]\ge c' t^{-1} \log t,\ \ t\ge 2.$$
 \end{enumerate}
   \end{cor}

Finally,  we have the following result on the weak convergence of $t \W_2(\mu_t^B,\mu)^2$.

\beg{thm}\label{T1.4}  Let $B\in  \BB^\aa$    for  some $\aa\in [0,1]$, and let $\pp M$ be empty or convex. If $d<2(1+\aa)$, then
$$\lim_{t\to\infty} \sup_{x\in M} |\P^x (t \W_2(\mu_t^B,\mu)^2<a)- F(a)|=0,\ \ a\ge 0,$$
where $F(a):=\P(\xi<a)$  for
$$\xi:=\sum_{i=1}^\infty\ff{2\xi_i^2}{\ll_i B(\ll_i)}$$
and i.i.d. random variables  $\{\xi_i\}$ with the standard normal distribution $N(0,1)$. \end{thm}

Following the line of   \cite{eWZ},  we will first study the modified empirical measure $\mu_{t, r}^B:= \mu_t^B  P_r$ for $r>0$ in Section 2,  present some lemmas in Section 3,
and finally prove  Theorems \ref{T1.1},  \ref{T1.2}  and \ref{T1.4} in Sections 4,  5 and 6 respectively.

 \section{Modified empirical measures}

In this part, we allow $M$ to be non-compact, but assume that  the (Neumann)   semigroup  $ P_t$ generated by $L$ is ultracontractive, i.e.
\beq\label{ULT} \| P_t\|_{1\to\infty}:=\sup_{\mu(|f|)\le 1} \| P_t f\|_\infty<\infty,\ \ t>0.\end{equation}
Consequently,  $-L$ has discrete spectrum and the heat kernel $ p_t(x,y)$ of $ P_t$ with respect to $\mu$ satisfies
\beq\label{HK0}  p_t(x,y)=1+ \sum_{i=1}^\infty \e^{-t\ll_i}\phi_i(x)\phi_i(y)\le \| P_t\|_{1\to\infty}<\infty,\ \ t>0, x,y\in M,\end{equation}
where $\{\ll_i\}_{i\ge 0} $ are all eigenvalues of $-L$ and
 $\{\phi_i\}_{i\ge 0} $ is  the eigenbasis, i.e. $\phi_0\equiv 1$ and $\{\phi_i\}_{i\ge 0}$ is an orthonormal basis of $L^2(\mu)$ with $L\phi_i= -\ll_i \phi_i$.

For any $p\ge 1$ and $f\in L^p(\mu)$,   let $\|f\|_p:=\{\mu(|f|^p)\}^{\ff 1 p}$ be the $L^p(\mu)$-norm of $f$.
 Then there exists a constant $c>0$ such that
 $$\|P_t f\|_p\le c \e^{-\ll_1t} \|f\|_p,\ \ t\ge 0, p\in [1,\infty],  f\in L^p_0(\mu), $$
 where $L_0^p(\mu):=\{f\in L^p(\mu): \mu(f):= \int_Mf\d\mu=0\}$.
 Consequently, for any $B\in \BB$,
 \beq\label{EXP} \|P_t^B f\|_p= \|\E P_{S_t^B} f\|_p \le c \|f\|_p \E\e^{-\ll_1 S_t^B} = c\e^{-B(\ll_1)t} \|f\|_p,\ \   t\ge 0,p\in [1,\infty], f\in L^p_0(\mu).\end{equation}

   As in \cite{eWZ}, we consider the modified empirical measure
 $$ \mu_{t,r}^B:= \mu_t^B  P_r,\ \ \ r,t>0.$$
 By \eqref{HK0}, we have
\beq\label{MTR} f_{t,r}:= \ff{ \d\mu_{t,r}^B}{\d\mu}
  = 1+ \ff 1 {\ss t} \sum_{i=1}^\infty \e^{-r\ll_i} \psi_i(t)\phi_i, \ \
  \psi_i(t):= \ff 1 {\ss t} \int_0^t \phi_i(X_s^B)\d s,\ \ \ \ r,t>0.\end{equation}
 The main result in this section is the following.

 \begin{thm}\label{T2.1} Let $B\in {\bf B}$, $M$ be a $d$-dimensional connected complete Riemannian manifold possibly with a boundary such that $\eqref{ULT}$ holds.
\beg{enumerate} \item[$(1)$] For any $r>0$,
$$
\lim_{t\to\infty}\sup_{x\in M}\left|t\E^x[\W_2(\mu^B_{t,r},\mu)^2]-\sum_{i=1}^\infty\frac{2}{\lambda_i B(\ll_i) \e^{2r\lambda_i}}\right|=0.
$$
\item[$(2)$]   For any $C>0$, let
$$\scr P(C):=\{\nu\in \scr P:\ \nu=h_\nu\mu, \|h_\nu\|_\infty\le C\}.$$
Then for any $C>1$,
$$\lim_{t\to\infty}\sup_{\nu\in \scr P(C)}\left|\P^\nu(t\W_2(\mu_{t,r}^B,\mu)^2<a)-F_r(a)\right|=0,\quad a\in \R,$$
where for i.i.d. random variables $\xi_i$ with distribution $N(0,1)$,  $F_r:= \P(\xi_r<\cdot)$ is the distribution function of
$$\xi_r:=\sum_{i=1}^\infty\ff {2\xi_i^2}{\ll_i B(\ll_i)\e^{2\ll_i r}},\quad r>0.$$\end{enumerate}
\end{thm}

 To prove this result, we first present some lemmas, where the first follows  from  \cite[Lemma 2.3]{eWZ}, which goes back to  \cite[Proposition 2.3]{eAST}.

\begin{lem}\label{ee:lem1} Let $B\in {\bf B}$, $\mathscr{M}(a,b):=\frac{a-b}{\log a-\log b}1_{\{a\wedge b>0\}}.$ Then
 $$\W_2(\mu^B_{t,r},\mu)^2\leq\int_M\frac{|\nabla L^{-1}(f_{t,r}-1)|^2}{\mathscr{M}(f_{t,r},1)}\,\d\mu,\quad t,r>0.$$ \end{lem}

By the ergodicity we have   $\lim_{t\to\infty} \mathscr{M}(f_{t,r},1) =1$ (see Lemma \ref{ee:lem3} below), so that  this lemma implies that  $t\W_2(\mu_{t,r}^B,\mu)^2$ is asymptotically bounded above by
\beq\label{XIR} \Xi_r(t):=t\mu\big( |\nabla L^{-1}(f_{t,r}-1)|^2\big), \ \  \quad t,r>0,\end{equation}  where $\mu(f):=\int_Mf\d\mu$ for $f\in L^1(\mu)$.
Thus, we first estimate $\Xi_r(t)$.

\begin{lem}\label{ee:lem2} Let $B\in {\bf B}$.
There exists a constant $c>0$ such that
\begin{equation}\label{ee:equ01}
\left|\E^{\nu}\Xi_r(t)-\sum_{i=1}^{\infty}\frac{2}{\ll_iB(\ll_i)\e^{2r\lambda_i}}\right|\leq\frac{c\|h_{\nu}\|_{\infty}}{t}\sum_{i=1}^{\infty}\frac{1}{\ll_iB(\ll_i)\e^{2r\lambda_i}},\quad t\geq 1,r>0,
\end{equation}
holds for any probability measure $\nu=h_{\nu}\mu.$ Consequently,
\begin{equation}\label{ee:equ02}\sup_{x\in M}\left|\E^x\Xi_r(t)-\sum_{i=1}^{\infty}\frac{2}{\ll_iB(\ll_i)\e^{2r\lambda_i}}\right|\leq\frac{c\| P_{\frac{r}{2}}\|_{2\to\infty}^2}{t}\sum_{i=1}^\infty\frac{1}{\ll_iB(\ll_i)\e^{r\lambda_i}},\quad t\geq 1,r>0.
\end{equation}
\end{lem}

\begin{proof}  By \eqref{HK0}, \eqref{MTR}, \eqref{XIR}, $L\phi_i=-\ll_i\phi_i$ and $\mu(\phi_i\phi_j)=1_{\{i=j\}}$ for $i,j\ge 0$,
we obtain
 \begin{equation}\label{ee:equ2}\Xi_r(t)=\sum_{i=1}^\infty\frac{|\psi_i(t)|^2}{\lambda_i \e^{2r\lambda_i}},\quad t,r>0.\end{equation}
Since  $P_t^B$ is the Markov semigroup of $X_t^B$,  the Markov property implies
$$\E^\nu(\phi_i(X_t^B)|X_s^B)=  P_{t-s}^B \phi_i(X_s^B)=\e^{-B(\ll_i) (t-s)}\phi_i(X_s^B),\ \ i\ge 0, t\ge s\ge 0.$$ So,   $\psi_i(t):=\ff 1 {\ss t}\int_0^t \phi_i(X_s^B)\d s$ satisfies
\begin{equation}\label{ee:equ3}\begin{split}
&\E^\nu|\psi_i(t)|^2=\frac{1}{t}\E^\nu\left|\int_0^t\phi_i(X^B_s)\,\d s\right|^2
=\frac{2}{t}\int_0^t\,\d s_1\int_{s_1}^t\E^\nu[\phi_i(X^B_{s_1})\phi_i(X^B_{s_2})]\,\d s_2\\
&=\frac{2}{t}\int_0^t\E^\nu|\phi_i(X^B_{s_1})|^2\,\d s_1\int_{s_1}^t \e^{-B(\ll_i)(s_2-s_1)}\,\d s_2\\
 &=\frac{2}{B(\ll_i) t}\int_0^t\nu( P_s^B\phi_i^2)(1-\e^{-B(\ll_i)(t-s)})\,\d s,\ \ t>0.\end{split}\end{equation}
This together with \eqref{ee:equ2} imply
\begin{equation}\label{ee:equ4}\E^\nu\Xi_r(t)=\frac{2}{t}\sum_{i=1}^\infty\frac{1}{\ll_iB(\ll_i)\e^{2r\lambda_i}}\int_0^t\nu( P_s^B\phi_i^2)(1-\e^{-B(\ll_i)(t-s)})\,\d s=: I_1+I_2,\end{equation}
where
\begin{equation}\begin{split}\label{ee:equ5}
I_1:=\frac{2}{t}\sum_{i=1}^\infty\int_0^t\frac{1-\e^{-(t-s)B(\ll_i)}}{\ll_iB(\ll_i)\e^{2r\lambda_i}}\,\d s
=\sum_{i=1}^\infty\frac{2}{\ll_iB(\ll_i)\e^{2r\lambda_i}}-\frac{2}{t}\sum_{i=1}^\infty\frac{1-\e^{-B(\ll_i) t}}{\lambda_i B(\ll_i)^2 \e^{2r\lambda_i}},
\end{split}\end{equation}
and due to $\nu( P^B_s\phi_i^2)=\mu(h_{\nu}  P_s^B\phi_i^2)=\mu(\phi_i^2  P_s^B h_{\nu})$,
$$I_2:=\E^\nu \Xi_r(t)-I_1=\frac{2}{t}\sum_{i=1}^{\infty}\int_0^t\frac{1-\e^{-(t-s)B(\ll_i)}}{\ll_iB(\ll_i)\e^{2r\lambda_i}}\mu(\phi_i^2 P_s^B h_{\nu}-1)\,\d s.$$
Since  $\mu(\phi_i^2)=1,$  by \eqref{EXP},    there exists a constant $c_0>0$ such that
$$|\mu(\phi_i^2 P_s^B h_\nu-1)|=|\mu(( P_s^B h_\nu-1)\phi_i^2)|\leq\| P_s^B(h_\nu-1)\|_\infty\le  c_0 \e^{-B(\ll_1) s}\|h_\nu\|_\infty,\quad s\geq 0.$$
Therefore, we find a   constant $c_1>0$ such that
\begin{equation}\label{ee:equ05}|I_2|\leq\frac{c_1}{t}\|h_\nu\|_\infty\sum_{i=1}^\infty\frac{1}{\ll_iB(\ll_i)\e^{2r\lambda_i}}<\infty.\end{equation}
Combining \eqref{ee:equ4}, \eqref{ee:equ5} and \eqref{ee:equ05}, we find  a constant $c_2>0$ such that
$$\left|\E^\nu \Xi_r(t)-\sum_{i=1}^\infty\frac{2}{\ll_iB(\ll_i)\e^{2r\lambda_i}}\right|\leq\frac{c_2\|h_\nu\|_\infty}{t}\sum_{i=1}^\infty\frac{1}{\ll_iB(\ll_i)\e^{2r\lambda_i}}.$$
When $\nu=\delta_x$, \eqref{ee:equ4} becomes
\begin{equation}\label{ee:equ6}\E^x\Xi_r(t)\leq I_1+I_2(x),\end{equation}
where $I_1$ is in \eqref{ee:equ5} and
$$I_2(x):=\frac{2}{t}\sum_{i=1}^\infty\int_0^t\frac{1-\e^{-B(\ll_i)(t-s)}}{\ll_iB(\ll_i)\e^{2r\lambda_i}} P_s^B  \{\phi_i^2-1\}(x)\,\d s.$$
Since $\mu(\phi_i^2)=1$,  \eqref{EXP} implies $\|P_s^B \phi_i^2-1\|_\infty \le c \e^{-B(\ll_1) s} \|\phi_i\|_\infty^2.$  Combining this with
$$\|\phi_i\|_\infty^2 =\e^{r\ll_i }\|P_{\ff r 2} \phi_i\|_\infty^2\le \e^{\ll_i r} \|P_{\ff r 2}\|_{2\to\infty}^2,$$
  we find a constant $c_3>0$ such that
$$  I_2(x)\leq\frac{2}{t}\sum_{i=1}^\infty\int_0^t\frac{c}{\ll_iB(\ll_i)\e^{r\lambda_i}} \e^{-B(\ll_1) s}  \|P_{\ff r 2}\|_{2\to \infty}^2\d s
\le \ff{c_3\|P_{\ff r 2}\|_{2\to\infty}^2}t \sum_{i=1}^\infty \frac{1}{\ll_iB(\ll_i)\e^{r\lambda_i}}.$$  This together with \eqref{ee:equ5} and \eqref{ee:equ6}  implies \eqref{ee:equ02}.
\end{proof}

The following Lemma shows that $\lim_{t\to\infty}\mathscr{M}(f_{t,r},1)=1,r>0$.

\begin{lem}\label{ee:lem3}
Let $\|f_{t,r}-1\|_{\infty}=\sup_{y\in M}|f_{t,r}(y)-1|$. Then there exists a function $c:\mathbb{N}\times(0,\infty)\to(0,\infty)$ such that
$$\sup_{x\in M}\E^x[\|f_{t,r}-1\|_\infty^{2k}]\leq c(k,r)t^{-k},\quad t\geq 1,r>0,k\in\mathbb{N}.$$
\end{lem}
\begin{proof}
For fixed $r>0$ and $y\in M$, let $f= p_r(\cdot,y)-1$. For any $k\in\mathbb{N}$, we consider
$$I_k(s):=\E^x\left|\int_0^sf(X^B_t)\,dt\right|^{2k}=(2k)!\E^x\int_{\Delta_k(s)}f(X^B_{s_1})\cdots f(X^B_{s_{2k}})\,\d s_1\cdots \d s_{2k}, $$
where $\Delta_k(s):=\{(s_1,\cdots,s_{2k})\in[0,s]:0\leq s_1\le \cdots\leq s_{2k}\leq s\}$.

By the proof of  \cite[Lemma~2.5]{eWZ} with $X_t^B$ replacing $X_t$, we obtain
\begin{equation}\label{ee:equ8}I_k(t)\leq\sup_{s\in[0,t]}I_{k}(s)\leq \{2k(2k-1)\}^k\left(\int_{\Delta_1(t)}(\E^x|g(r_1,r_2)|^k)^{\frac{1}{k}}\,\d r_1\,\d r_2\right)^k,\end{equation}
where $g(r_1,r_2)=(f P^B_{r_2-r_1}f)(X^B_{r_1}),r_2\geq r_1\geq 0$.\\
By \eqref{ULT} we have
$$\|f\|_\infty=\| p_r(\cdot,y)-1\|_\infty\leq 2\| P_r\|_{1\to\infty}<\infty,$$
which together with \eqref{EXP}  implies
\begin{equation*}\begin{split}
|g(r_1,r_2)|^k\leq\|f P^B_{r_2-r_1}f\|_\infty^k\leq c \e^{-B(\ll_1) (r_2-r_1)k}\|f\|_\infty^{2k}
\leq c_1\| P_r\|_{1\to\infty}^{2k}\e^{-B(\ll_1) (r_2-r_1)k}
\end{split}\end{equation*} for some constant $c_1>0$.
Thus, there exists a constant $c_2>0$ such that
\begin{equation*}\begin{split}
 \int_{\Delta_1(t)}(\E^x |g(r_1,r_2)|^k)^\frac{1}{k}\,\d r_1\,\d r_2 \leq \int_0^t\,\d r_1\int_{r_1}^t c_1\| P_r\|^2_{1\to\infty}\e^{-B(\lambda_1)(r_2-r_1)}\,\d r_2
\le c_2 \| P_r\|_{1\to\infty}^{2} t.
\end{split}\end{equation*}
Combining this with    \eqref{ee:equ8}, we find a constant $c_3>0$ such that
$$
\sup_{x,y\in M}\E^x[|f_{t,r}(y)-1|^{2k}]=t^{-2k}I_k(t)\leq c_3 \| P_r\|_{1\to\infty}^{2k} t^{-k},\quad t\geq 1,r>0.$$
Noting that $f_{t,r}-1= P_{\ff r 2}(f_{t,\ff r 2}-1)$, this implies that for some constant $c>0$
 \begin{equation*}\begin{split}
&\sup_{x\in M}\E^x[\|f_{t,r}-1\|_\infty^{2k}]=\sup_{x\in M}\E^x[\| P_{\frac{r}{2}}(f_{t,\frac{r}{2}}-1)\|_\infty^{2k}]\\
&\leq\| P_{\frac{r}{2}}\|_{2k\to\infty}^{2k}\sup_{x\in M}\E^x[\mu(|f_{t,\frac{r}{2}}-1|^{2k})]
\leq c\| P_{\frac{r}{2}}\|_{1\to\infty}^{4k}t^{-k}.
\end{split}\end{equation*}
 \end{proof}

\beg{proof}[Proof of Theorem \ref{T2.1}]  (1) It suffices to verify the following estimates for any $r>0$:
\beq\label{LPB}\liminf_{t\to\infty}\inf_{x\in M}\{t\E^x[\W_2(\mu^B_{t,r},\mu)^2]\}\ge \sum_{i=1}^\infty\frac{2}{\ll_iB(\ll_i)\e^{2r\lambda_i}}, \end{equation}
 \beq\label{UPB}
 \limsup_{t\to\infty}\sup_{x\in M}\{t\E^x[\W_2(\mu^B_{t,r},\mu)^2]\}\leq\sum_{i=1}^\infty\frac{2}{\ll_iB(\ll_i)\e^{2r\lambda_i}}. \end{equation}

  Let $B_\si:= \{\|f_{t,r}-1\|_\infty\le\si^{\ff 2 3}\}$ for $\si>0$. By the proofs of \cite[(2.53) and (2.54)]{eWZ} for $X_t^B$ replacing $X_t$, there exists a constant $c>0$ such that
 \beq\label{*WW*} t \W_2(\mu_{t,r}^B,\mu)^2\ge 1_{B_\si} \big\{\Xi_r(t)-ct\si^{\ff 5 3}\big\},\ \ r,t,\si>0.\end{equation}
 Taking   $\si= t^{-\gg}$ for some $\gg\in (\ff 3 5,\ff 3 4)$, we have $t\si^{\ff 5 3}\downarrow 0$ as $t\uparrow \infty$, and  according to Lemma \ref{ee:lem3},
 $$\lim_{t\to\infty} \sup_{x\in M}\P^x(B_\si^c)\le \lim_{t\to\infty} \sup_{x\in M}t^{\ff{4\gg}3} \E^x[\|f_{t,r}-1\|_\infty^2]=0, $$    so that by \eqref{ee:equ02}
$$\limsup_{t\to\infty} \sup_{x\in M} \E^x [1_{B_\si^c} \Xi_r(t)] \le c(r) \limsup_{t\to\infty} \sup_{x\in M}\P^x(B_\si^c) =0,$$
where $c(r):= \sum_{i=1}^\infty\ff 2 {\ll_i B(\ll_i) \e^{2\ll_i r}}<\infty.$ Thus, \eqref{*WW*} yields
$$ \liminf_{t\to\infty} \inf_{x\in M} \E^x \big[t\W_2(\mu_{t,r}^B, \mu)^2\big]\ge  \liminf_{t\to\infty} \inf_{x\in M} \E^x \big[\Xi_r(t)\big],$$
which together with \eqref{ee:equ02} implies \eqref{LPB}.

  Since $\mu(\phi_i^2)=1$ and $\ll_1>0$, by taking $x=y$ in \eqref{HK0} and integrating with respect to $\mu(\d x)$,  we obtain
$$ \sum_{i=1}^\infty\frac{1}{\ll_iB(\ll_i)\e^{2r\lambda_i}} \leq \ff 1 {\ll_1B(\ll_1)  } \sum_{i=1}^\infty \e^{-2r\lambda_i}<\infty.$$
 For any $\eta\in(0,1)$, let
$$A_\eta=\{\|f_{t,r}-1\|_\infty\leq \eta\}.$$
Noting that $f_{t,r}(y)\geq 1-\eta$ implies
$$\mathscr{M}(1,f_{t,r}(y))\geq\sqrt{f_{t,r}(y)}\geq\sqrt{1-\eta},$$
by Lemma \ref{ee:lem1} and \eqref{ee:equ02}, we find a constant   $c(r)>0$ such that
\begin{equation*}\begin{split}
&t\sup_{x\in M}\E^x[1_{A_\eta}\W_2(\mu_{t,r}^B,\mu)^2]\leq\sup_{x\in M}\E^x\left\{\frac{\Xi_r(t)}{\sqrt{1-\eta}}\right\}\\
&\leq\frac{1}{\sqrt{1-\eta}}\sum_{i=1}^\infty\frac{2}{\ll_iB(\ll_i)\e^{2r\lambda_i}}\left(1+\frac{c(r)}{t}\right),\quad t>0,\eta\in(0,1).
\end{split}\end{equation*}
Thus,
\begin{equation}\label{ee:equ10}\begin{split}
&t\sup_{x\in M}\E^x[\W_2(\mu_{t,r}^B,\mu)^2]\\
&\leq\frac{1}{\sqrt{1-\eta}}\sum_{i=1}^\infty\frac{2}{\ll_iB(\ll_i)\e^{2r\lambda_i}}\left(1+\frac{c(r)}{t}\right)
 +t\sup_{x\in M}\E^x[1_{A_{\eta}^c}\W_2(\mu_{t,r}^B,\mu)^2]\\
&\leq\frac{1+c(r)t^{-1}}{\sqrt{1-\eta}}\sum_{i=1}^\infty\frac{2}{\ll_iB(\ll_i)\e^{2r\lambda_i}}+t\sup_{x\in M}\sqrt{\Pp^x(A_\eta^c)\E^x[\W_2(\mu_{t,r}^B,\mu)^4]}.
\end{split}\end{equation}
As shown in the proof  of \cite[Proposition~2.6]{eWZ}, we have
\begin{equation}\label{ee:equ11}\E^x\W_2(\mu_{t,r}^B,\mu)^4\leq\| P_r\|_{1\to\infty}(\mu\times \mu)(\rho^4)<\infty.\end{equation}
Moreover, Lemma \ref{ee:lem3} implies that for some constant $c(k,r)>0$
$$\sup_{x\in M}\Pp^x(A_\eta^c)\leq \eta^{-2k}c(k,r)t^{-k}.$$
  By taking $k=4$ and applying \eqref{ee:equ10} and \eqref{ee:equ11}, we conclude that
$$\limsup_{t\to\infty}\Big\{t\sup_{x\in M}\E^x[\W_2(\mu_{t,r}^B,\mu)^2]\Big\}\leq\frac{1}{\sqrt{1-\eta}}\sum_{i=1}^\infty\frac{2}{\ll_iB(\ll_i)\e^{2r\lambda_i}}.$$
Then \eqref{UPB} follows by letting $\eta\to 0.$

(2)  By Lemma \ref{ee:lem2}, it suffices to prove that for any $C>1$
\begin{equation}\label{P261} \lim_{t\to\infty}\sup_{\nu\in\scr{P}(C)}|\P^\nu(\Xi_r(t)<a)-\P(\xi_r<a)|=0,\quad a\ge 0.\end{equation}

Recall that
$$\Xi_r(t)=\sum_{i=1}^\infty \ff {|\psi_i(t)|^2} {\ll_i \e^{2\ll_i r}},\quad t,r>0.$$
Define for any $n\ge 1$,
$$\Psi_n(t):=(\psi_1(t),\cdots,\psi_n(t)),\quad t>0.$$
Then, for any $\vartheta\in \R^n$, we have
$$\langle \Psi_n(t),\vartheta\rangle=\ff 1 {\sqrt{t}}\int_0^t (\sum_{i=1}^n\vartheta_i \phi_i(X_s^B))ds.$$
By \cite[Theorem 2.4$'$]{eWL}, when $t\to\infty$, the law of $\langle \Psi_n(t),\vartheta\rangle$ under $\P^\nu$ converges weakly to the Gaussian distribution $N(0,\si_{n,\vartheta})$ uniformly in $\nu\in\scr{P}(C)$ with variance
\beg{equation*}\begin{split}
&\si_{n,\vartheta}:=\lim_{t\to\infty}\E^\mu\langle\Psi_n(t),\vartheta\rangle^2\\
&=\lim_{t\to\infty}\ff 2 t\sum_{i=1}^n\vartheta_i^2\int_0^t \d s_1\int_{s_1}^t \e^{-B(\ll_i)(s_2-s_1)}\d s_2=\sum_{i=1}^n\ff {2\vartheta_i^2}{B(\ll_i)}.
\end{split}\end{equation*}
Thus, for any $\vartheta\in\R^n$,
$$\lim_{t\to\infty}\E^\nu \e^{\i\langle \Psi_n(t),\vartheta\rangle}=\int_{\R^n}\e^{\i\langle x,\vartheta \rangle}\prod_{i=1}^n N(0,2B(\ll_i)^{-1})(\d x_i)\,\text{ uniformly\ in\ }  \nu\in \scr{P}(C),$$
so that the distribution of $\Psi_n(t)$ under $\P^\nu$ converges weakly to $\prod_{i=1}^n N(0,2B(\ll_i)^{-1})$ as $t\to\infty$. Therefore, letting
$$\Xi_r^{(n)}(t):=\sum_{i=1}^n\ff{|\psi_i(t)|^2}{\ll_i B(\ll_i)\e^{2\ll_i r}},\quad \xi_r^{(n)}:=\sum_{i=1}^n\ff{2\xi_i^2}{\ll_i B(\ll_i)\e^{2\ll_i r}},$$
we have
\beq\label{LKK} \lim_{t\to\infty}\sup_{\nu\in\scr{P}(C)}|\P^\nu(\Xi_r^{(n)}(t)<a)-\P(\xi_r^{(n)}<a)|=0,\quad a\ge 0.\end{equation}

On the other hand, by $\eqref{ee:equ2}$ and $\eqref{ee:equ3}$, we find some constant $C_1>0$ such that
\begin{equation*}\begin{split}
&\sup_{\nu\in\scr{P}(C)}\E^\nu|\Xi_r(t)-\Xi_r^{(n)}(t)|\\
&=\ff 2 t \sup_{\nu\in\scr{P}(C)}\sum_{i=n+1}^\infty\ff{\e^{-2\ll_i r}}{\ll_i B(\ll_i)}\int_0^t\nu(P_s^B\phi_i^2)(1-\e^{-B(\ll_i)(t-s)})\d s\le C_1\vv_n,
\end{split}\end{equation*}
where $\vv_n:=2\sum_{i=n+1}^\infty\ff 2 {\ll_i B(\ll_i) \e^{2\ll_i r}}\to 0$ as $n\to\infty$. This together with \eqref{LKK} implies \eqref{P261}.
\end{proof}

    \section{Some lemma}
From now on, we assume that $M$ is compact.  For any $q\ge p\ge 1$, let $\|\cdot\|_{p\to q}$ be the operator norm from $L^p(\mu)$ to $L^q(\mu)$. When $p=q$, we simply denote $\|\cdot\|_p=\|\cdot\|_{p\to p}.$
Then  there exist constants $\kappa,\lambda>0$ such that
\begin{equation}\label{ee:itr71}
\| {P}_t-\mu\|_{p\to q} \leq\kappa(1\wedge t)^{-\frac{d}{2}(p^{-1}-q^{-1})}\e^{-\ll_1 t},\quad t>0,q\geq p\geq 1.
\end{equation}
 Next, by the triangle inequality of $\W_2$, we obtain
\begin{equation}\label{ewine}\E[\W_2(\mu_t^B,\mu)^2]\leq(1+\varepsilon)\E[\W_2(\mu_{t,r}^B,\mu)^2]+(1+\varepsilon^{-1})\E[\W_2(\mu_t^B,\mu_{t,r}^B)^2],\quad \varepsilon>0.\end{equation}
As shown in \cite{eWZ} for $B(\ll)=\ll$ that, to prove Theorem \ref{T1.1}, we need to estimate $\E[\W_2(\mu_t^B,\mu_{t,r}^B)^2]$ and to refine the estimate on $\E[\W_2(\mu_{t,r}^B,\mu)^2]$ for compact $M$.
These are included in   the following  lemmas.

\begin{lem}\label{lemma31} Let $B\in \BB$ and
 $\mu_{t,r,\varepsilon}^B=(1-\varepsilon)\mu_{t,r}^B+\varepsilon\mu,\varepsilon\in[0,1]$.
There exists a constant~$c>0$ such that
\begin{equation}\label{lem310}
\E^\nu[\W_2(\mu_t^B,\mu_{t,r}^B)^2]\leq c\|h_\nu\|_\infty r,\quad \nu=h_\nu\mu,
\end{equation}
\begin{equation}\label{lem311}
\W_2(\mu_{t,r,\varepsilon}^B,\mu_{t,r}^B)^2\leq c\varepsilon,\quad t,r\geq 0,\varepsilon\in[0,1].
\end{equation}
\end{lem}

\begin{proof}
Since for $t>0$,
$$\pi_t( \d x,dy):=\left(\frac{1}{t}\int_0^tp_r(x,y)\delta_{X^B_{s}} ( \d x)\,\d s\right)\mu( \d y)\in\mathscr{C}(\mu_t^B,\mu_{t,r}^B),$$
we have
\begin{equation}\label{lem313}\begin{split}
&\W_2(\mu_{t,r}^B,\mu_t^B)^2\leq\int_M\rho(x,y)^2\,\pi_t( \d x, \d y)\\
&=\frac{1}{t}\int_0^t\,\d s\int_M  p_r(X^B_s,y)\,\rho(X^B_s,y)^2\mu( \d y).
\end{split}\end{equation}
Since $\nu=h_\nu\mu$, by the $P_t^B$-invariance of $\mu$, we find a constant $c_1>0$ such that
\begin{equation}\begin{split}
&\E^\nu\int_M  p_r(X^B_s,y)\rho(X^B_s,y)^2\mu( \d y)\leq\|h_\nu\|_\infty\mu\left[P_s^B\left(\int_M p_r(x,y)\rho(\cdot,y)^2\mu( \d y)\right)\right]\\
&=\|h_\nu\|_\infty\E^\mu[\rho(X_0,X_r)^2]\leq c_1\|h_\nu\|_\infty r,\ \ s\ge 0,
\end{split}\end{equation}
where   the last step is due to \cite[Lemma 3.1]{eWZ}. Substituting this into \eqref{lem313}, we prove \eqref{lem310}.

On the other hand, let  $D$ be  the diameter of $M$. Since
$$\pi( \d x, \d y):=(1-\varepsilon)\mu_{t,r}^B( \d x)\delta_x( \d y)+\varepsilon\mu( \d x)\mu_{t,r}^B( \d y)\in\mathscr{C}(\mu_{t,r,\varepsilon}^B,\mu_{t,r}^B),$$
we obtain
$$\W_2(\mu_{t,r,\varepsilon}^B,\mu_{t,r}^B)^2\leq\int_{M\times M}\rho(x,y)^2\pi( \d x, \d y)\leq\varepsilon D^2,\ \ t,r>0, \vv\in [0,1].$$
Then the proof is finished.
\end{proof}


\begin{lem}\label{lem32} Let $B\in \BB^\aa$ for some $\aa\in [0,1],$ and let   $d<2(1+\alpha)$.
\beg{enumerate} \item[$(1)$] For any $q\in(\frac{d}{2}\vee 1,\frac{d}{d-2\alpha})$, there exists a constant $c>0$, such that
\begin{equation}\label{lem320}
\sup_{y\in M}\E^\mu[|f_{t,r}(y)-1|^2]\leq\frac{c}{t r^{\frac{d}{2q}}},\quad t\geq 1,r\in(0,1].
\end{equation}
 \item[$(2)$]  For any  $q\in(\frac{d}{2}\vee 1,\frac{d}{d-2\alpha})$ and $\gamma\in (1,\frac{2q}{d})$,
\begin{equation}\label{lem331}\lim_{t\to\infty}\sup_{y\in M}\E^\mu[|\mathscr{M}((1-t^{-\gg})f_{t,t^{-\gg}}(y)+t^{-\gg},1)^{-1}-1|^p]=0,\quad p>0.\end{equation}
\end{enumerate}
\end{lem}

\begin{proof} (1) For fixed $y\in M$, simply denote   $f = p_r(\cdot,y)-1$.   Then
\beq\label{*PP} \E^\mu \big[|f_{t,r}-1|^2\big] =\ff {2  } {t^2} \E^\mu \int_{0}^t f(X_{r_1}^B)\d r_1 \int_{r_1}^t   f(X_{r_2}^B) \d r_2,\end{equation}
 Since  $P_t^B$   is invariant with respect to $\mu$, we obtain
\begin{equation}\label{lem321}\begin{split}
&\E^\mu[f(X^B_{r_1})f(X^B_{r_2})]=\mu\big(P^B_{r_1}(fP^B_{r_2-r_1}f)\big) =\mu(f  P^B_{r_2-r_1}f)\\
  &\leq \|f\|_{\ff q {q-1}}\|P_{r_2-r_1}^B f\|_q \le \|f\|_{\ff q {q-1}}  \|P^B_{\frac{r_2-r_1}{2}}\|_{1\to q}\|P^B_{\frac{r_2-r_1}{2}}f\|_1,\ \ r_2>r_1\ge 0.
\end{split}\end{equation}
By   $f= p_r(\cdot,y)-1$ and \eqref{ee:itr71}, we find some constants $c_1>0$ such that
\begin{equation}\label{lem3211}\begin{split}
\|f\|_{\ff q {q-1}}\leq 1+\| p_{r}(\cdot,y)\|_{\ff q {q-1}}\leq 1+\| P_{\ff r {2}}\|_{1\to \ff q {q-1}}\leq c_1 r^{-\frac{d}{2q}},\quad r\in(0,1],q\geq 1.
\end{split}\end{equation}
Moreover, since $P_t^B$ is the semigroup of $X_t^B:=X_{S_t^B},$ by \eqref{ee:itr71} and noting that $B\in \BB^\aa$ implies
\beq\label{BR} B(r)\ge k_0 (r\land r^\aa)\ge k_1 r^\aa -k_2,\ \ r\ge 0\end{equation} for some constants $k_0, k_1,k_2>0$, we find a constant $c_2>0$ such that
\begin{align*}
&\|P^B_{r}\|_{1\to q}\leq \E \| P_{S_r^B} \|_{1\to q} \le c  \E\big[ (1\wedge S_r^B)^{-\frac{d(q-1)}{2q}} \big]\\
&\le c + c  \E\big[   (S_r^B)^{-\frac{d(q-1)}{2q}} \big] = c +\ff c {\GG(\frac{d(q-1)}{2q})} \int_0^\infty t^{\frac{d(q-1)}{2q}-1} \e^{-r B(t)}\d t \\
& \leq c_2(r^{-\frac{d(q-1)}{2\aa q}}+1),\ \ r>0.
\end{align*}
Since $\ff{d(q-1)}{2\aa q} <1$, by combining this with \eqref{EXP}, \eqref{*PP} and  \eqref{lem3211}, we find   constants $c_3,c_4>0$ such that
 \begin{align*}
 \E^\mu[|f_{t,r}(y)-1|^2]
\leq \ff{c_3}{r^{\ff d{2q} } t^2} \int_0^t \d r_1\int_{r_1}^t  ((r_2-r_1)^{-\ff{d(q-1)}{2\aa q}}+1)\e^{-\lambda_1^\alpha(r_2-r_1)}\,\d r_2
 \le \ff{c_4}{r^{\ff d{2q} } t},\ \ t,r>0.
\end{align*}

(2) Let $\theta>0$ be small enough such that  $\gamma(\frac{d}{2q}+\frac{\theta p}{2})< 1$.
According to the proof of the \cite[Lemma 3.3]{eWZ},    there exists a map $C: (0,1)\to (0,\infty)$ such that
\begin{align*}
&\sup_{y\in M}\E^\mu[|\mathscr{M}((1-r)f_{t,r}(y)+r,1)^{-1}-1|^p]\leq\delta_\eta+(1+\theta^{-1}r^{-\frac{\theta}{2}})^p\sup_{y\in M}\P^\mu(\{|f_{t,r}(y)-1|> \eta\})\\
&\leq\delta_\eta+C(\eta)  t^{-1}r^{-\frac{d}{2p}-\frac{\theta p}{2}},\ \ t\ge 1, r,\eta \in (0,1),
\end{align*}
holds for
$\delta_\eta=\left|\frac{1}{\sqrt{1-\eta}}-\frac{2}{2+\eta}\right|^q,\ \eta\in(0,1).$
This implies \eqref{lem331} by taking $r= t^{-\gg}$ and letting first $t\to\infty$ then $\eta\to 0.$ \end{proof}

\begin{lem}\label{lem34} Let $B\in \BB^\aa$ for some $\aa\in [0,1].$
For any  $p\in[1,2]$,    there exists a constant $c>0$ such that
\begin{equation}\label{lem341}
\E^\mu[|\psi_i(t)|^{2p}]\leq c \lambda_i^{\alpha(p-2)+(p-1)(\frac{d}{2}-2\alpha)},\ \ i\ge 1.
\end{equation}
\end{lem}

\begin{proof} Since $P_t^B\phi_i=\e^{-B(\ll_i) t}\phi_i$, we have
 \begin{equation}\label{lem342}
g(r_1,r_2):=(\phi_i P^B_{r_2-r_1}\phi_i)(X^B_{r_1})=\e^{-(r_2-r_1)B(\lambda_i)}\phi_i(X^B_{r_1})^2.
\end{equation}
By \eqref{ee:equ8},    $\mu (P^B_{r_1} \phi_i^2) =\mu(\phi_i^2)=1$ and \eqref{BR},  we find a constant $c_1>0$ such that
\begin{equation}\label{lem343}\begin{split}
&t\E^\mu[|\psi_i(t)|^2]\leq c_1\int_0^t \d r_1\int_{r_1}^t\E^\mu[g(r_1,r_2)]\,\d r_2\\
&=c_1\int_0^t \d r_1\int_{r_1}^t \e^{-(r_2-r_1)B(\ll_i)}\mu(P^B_{r_1}\phi_i^2)\,\d r_2\leq   \frac{c_1t}{\lambda_i^\alpha},\quad t\geq 1,i\in\mathbb{N}.
\end{split}\end{equation}
On the other hand, by \eqref{ee:equ8}, \eqref{BR} and \eqref{lem342}, we find a constant $c_2>0$ such that
\begin{equation}\label{lem344}\begin{split}
&t^2\E^\mu[|\psi_i(t)|^4]\leq c_2 \left(\int_0^t \d r_1\int_{r_1}^t(\E^\mu[|g(r_1,r_2)|^2])^{\frac{1}{2}}\,\d r_2\right)^2\\
&\le c_2\left(\int_0^t \d r_1\int_{r_1}^t \e^{-(r_2-r_1)\lambda_i^\alpha}\sqrt{\mu( P^B_{r_1}\phi_i^4)}\,\d r_2\right)^2.
\end{split}\end{equation}
Moreover, \eqref{ee:itr71} and  $P_t\phi_i=\e^{-\lambda_i t}\phi_i$ yield
$$
\|\phi_i\|_\infty=\inf_{t>0}\{\e^{\lambda_i t}\|P_t\phi_i\|_\infty\}\leq\inf_{t>0}\{\e^{\lambda_it}\|P_t\|_{2\to\infty}\}\leq c_3\lambda_i^{\frac{d}{4}},\quad i\geq 1
$$ for some constant $c_3>0$, so that
 $$
\sqrt{\mu(P_r^B\phi_i^4)}=\sqrt{\mu(\phi_i^4)}\leq\sqrt{\|\phi_i\|_\infty^2\mu(\phi_i^2)}
 \leq c_3\lambda_i^{\frac{d}{4}},\quad i\geq 1.
$$
This together with  \eqref{lem344} implies  that for some   constant $c_4>0$
$$\E^\mu[|\psi_i(t)|^4]\leq c_4  \lambda_i^{\frac{d}{2}-2\alpha},\ \ i\ge 1.$$
Combining this with  \eqref{lem343} and using H\"{o}lder's inequality, we find a constant $c_5>0$ such that
\begin{align*}
&\E^\mu[|\psi_i(t)|^{2p}]=\E^\mu[|\psi_i(t)|^{4-2p}|\psi_i(t)|^{4(p-1)}]\\
&\leq(\E^\mu[|\psi_i(t)|^{2}])^{2-p}(\E^\nu[|\psi_i(t)|^{4}])^{p-1}\leq c_5 \lambda_i^{\alpha(p-2)+(p-1)(\frac{d}{2}-2\alpha)}.
\end{align*}
 \end{proof}

\begin{lem}\label{lem35} Let $B\in \BB^\aa$ for some $\aa\in [0,1].$
If $d< 2(1+\alpha)$, then there exists a constant $p>1$ such that
$$\limsup_{t\to\infty}\sup_{r>0 }\bigg\{t^p\E^\mu\int_M|\nabla L^{-1}(f_{t,r}-1)|^{2p}\,d\mu\bigg\}<\infty.$$
\end{lem}

\begin{proof} According to the proof of \cite[Lemma 3.5]{eWZ}, for any $p>1$ and $\vv>p-1$, there exists a constant $c_1(p,\vv)>0$ such that
$$t^p\E^\mu\int_M|\nabla L^{-1}(f_{t,r}-1)|^{2p}\,d\mu
\leq c_1(p,\vv) \sum_{i=1}^\infty i^\varepsilon\lambda_i^{\frac{d(p-1)}{2}-1}  \E^\mu[|\psi_i(t)|^{2p}].$$
Combining this with Lemma \ref{lem34} and \eqref{*2},    we find a   constant $c_2(p,\vv) >0$ such that
\begin{equation}\label{lem353}
\E^\mu\int_M|\nabla L^{-1}(f_{t,r}-1)|^{2p}\,d\mu\leq c_2(p,\vv) t^{-p} \sum_{i=1}^\infty i^{\delta_{p,\varepsilon}}
\end{equation}
holds for
$$\delta_{p,\varepsilon}:=\varepsilon+\frac{2}{d}\left\{(p-1)\left(d-2\alpha\right)+\alpha p-(2\alpha+1)\right\}.$$
So, it remains to show that $\delta_{p,\varepsilon}<-1$ holds for some constants $p>1$ and $\vv>p-1$.  This follows from the fact that for $\vv>0$ and $p_\vv:=1 +\ff \vv 2 $ we have
$\vv>p_\vv-1$ and
$$\lim_{ \vv\downarrow 0} \dd_{p_\vv,\vv}= -\ff{2(1+\aa)} d<-1.$$
 \end{proof}

Finally, the following lemma reduces arbitrary initial values to initial distributions with bounded density.

\beg{lem}\label{LYP} Let $B\in \BB$ and   $p\in (0,2]$.  Then for any $\vv>0$,
\beg{align*} &\aa_\vv:= \|P_{\vv^2}^B\|_{1\to\infty}<\infty,\\
&\sup_{x\in M} \E^x\big[ \W_p(\mu_t^B,\mu)^{1\lor p} \big]\le (1+\vv)\sup_{\nu\in \scr P(\aa_\vv)}    \E^\nu \big[\W_p(\mu_t^B,\mu)^{1\lor p}\big] +\ff {\vv(1+\vv)D^{p}}t,\\
& \inf_{\nu\in \scr P(\aa_\vv)}    \E^\nu \big[\W_p(\mu_t^B,\mu)^{1\lor p}\big] \le (1+\vv) \inf_{x\in M} \E^x \big[\W_p(\mu_t^B,\mu)^{1\lor p}\big] +\ff {\vv(1+\vv)D^p}t,\ \ t>\vv^2,\end{align*}
where $D$ is the diameter of $M$.
\end{lem}

\beg{proof} There exists a constant $c>0$ such that
$$ \|P_t\|_{1\to\infty}\le c (1+ t^{-\ff d 2}),\ \ t>0.$$
 This together with \eqref{LT} and  $ \int_1^\infty r^{\ff d 2 -1}\e^{-t B(r)}\d r<\infty$ implies
 \beg{align*} &\|P_t^B\|_{1\to\infty} =\sup_{\mu(|f|)\le 1} \sup_{x\in M} |\E^x f(X_{S_t^B})|\le \E\|P_{S_t^B}\|_{1\to\infty}\le c +c \E (S_t^B)^{-\ff d 2}\\
 &=c + \ff{c}{\GG(\ff d 2)} \int_0^\infty r^{\ff d 2 -1} \e^{-t B(r)}\d r<\infty,\ \ t>0.\end{align*}
 Next, for any $x\in M$ and $\vv>0$, let $\nu_{x,\vv}$ be the distribution of $X_{\vv^2}^B$. Then
 $$\Big\|\ff{\d\nu_{x,\vv}}{\d \mu}\Big\|_\infty = \sup_{\mu(|f|)\le 1} |P_{\vv^2}^B f(x)|\le \|P_{\vv^2}^B \|_{1\to\infty}=\aa_\vv,$$ so that
 \beq\label{N*N}  \nu_{x,\vv}\in  \scr P(\aa_\vv),\ \ x\in M, \vv>0.\end{equation}
Let
$$\tt \mu^B_{\vv,t}:= \ff 1 t \int_{\vv^2}^{\vv^2+t} \dd_{X_s^B}\d s,\ \ t>0.$$ By Markov property, we have
 \beq\label{B*B} \E^x \big[\W_2(\tt\mu_{\vv,t}^B,\mu)^2\big]= \E^{\nu_{x,\vv}}\big[\W_2(\mu_t^B,\mu)^2\big],\ \ x\in M, t,\vv>0.\end{equation}
 Moreover, it is easy to see that for any $t>\vv^2>0$,
 $$\pi:=\ff 1 t \int_{\vv^2}^t \dd_{(X_s^B,X_s^B)}\d s +\ff 1 t \int_0^{\vv^2} \dd_{(X_s^B, X_{t+s}^B)}\d s\in \C(\mu_t^B, \tt\mu_{\vv,t}^B),$$
 so that
 $$|\W_p(\mu_t,\mu)- \W_p(\tt\mu_{\vv,t}^B,\mu)|^{1\lor p}\le\big\{ \W_p(\mu_t,\tt\mu_{\vv,t}^B)\big\}^{1\lor p}\le \int_{M\times M} \rr^p\d\pi \le \ff {\vv^2 D^p} t.$$
 Combining this with \eqref{N*N} and \eqref{B*B}, we obtain
\beg{align*} &\sup_{x\in M} \E^x \big[\W_p(\mu_t,\mu)^{1\lor p}\big]\le (1+\vv) \sup_{x\in M} \E^x\big[\W_p(\tt\mu_{\vv,t}^B,\mu)^{1\lor p}\big] +(1+\vv^{-1})\ff{\vv^2D^p}t\\
&\le (1+\vv)\sup_{\nu\in \scr P(\aa_\vv)}    \E^\nu \big[\W_p(\mu_t^B,\mu)^{1\lor p}\big] +\ff {\vv(1+\vv)D^p}t.\end{align*}
Similarly,  the last  estimate also holds.

  \end{proof}

\section{Proof of Theorem \ref{T1.1}}

\subsection{Proof of Theorem \ref{T1.1}(1)}   Since $M$ is compact and $V\in C^2(M)$, there exists a constant $K>0$ such that
$$\textnormal{Ric}_V:=\Ric -\Hess_V\geq -K,$$ where $\Ric$ is the Ricci curvature.

When  $\partial M$ is either convex or empty, then
\beq\label{WW0} \W_p(\mu,\nu  P_r)^2\leq \e^{2Kr}\W_p(\mu,\nu)^2,\quad r>0, p\ge 1,\end{equation}
see for instance \cite{W10,W14}.
Since $\mu_{t,r}^B=\mu_t^B P_r$, this and  \eqref{LPB} imply
\beg{align*} &\e^{2Kr}\liminf_{t\to\infty}\left\{t\inf_{x\in M} \E^x[\W_2(\mu,\mu_{t}^B)^2]\right\}\\
&\geq\liminf_{t\to\infty}\left\{t\inf_{x\in M}\E^x[\W_2(\mu,\mu_{t,r}^B)^2]\right\}\geq\sum_{i=1}^\infty\frac{2}{\lambda_i B(\ll_i)  \e^{2r\lambda_i}},\ \ r>0.\end{align*}
By letting $r\to 0$, we prove the desired estimate for    $c=1$.

When $\pp M$ is non-convex,  the desired inequality follows  by using  the following estimate due to \cite[Theorem 2.7]{eCTT} replacing \eqref{WW0}: there exist  constants $c,\ll>0$ such that
 $$c\W_2(\nu P_r, \mu)\le  \e^{\ll r} \W_2(\nu,\mu),\ \ \nu\in \scr P, r>0.$$

\subsection{Proof of Theorem \ref{T1.1}(2)}    It suffices to prove for $p\in (0,\aa)$.  The proof is modified from that of  the proof of  \cite[Theorem 1.1]{eWZ}, the only difference is that we have to use  $\W_p$ for $p\in (0,\aa)$ replacing $\W_1$, since in this case we have $ \E[(S_t^B)^p]<\infty$.

For any $t\geq 1$ and $N\in\mathbb{N}$, we consider $\mu_N^B:=\frac{1}{N}\sum_{i=1}^N\delta_{X^B_{t_i}}$, where ~$t_i:=\frac{(i-1)t}{N},1\leq i\leq N$.
By taking the Wasserstein coupling
$$\frac{1}{t}\sum_{i=1}^N\int_{t_i}^{t_{i+1}}\delta_{X^B_s}( \d x)\delta_{X^B_{t_i}}( \d y)\,\d s\in\mathscr{C}(\mu_t^B,\mu_N^B),$$
we obtain
$$\W_p(\mu_t^B,\mu_N^B)\le \frac{1}{t}\sum_{i=1}^N\int_{t_i}^{t_{i+1}}\rho(X^B_{s},X^B_{t_i})^p\,\d s.$$
By \cite[(3.6)]{eWZ}  that
$$\sup_{x\in M} \E^x \rr(X_0, X_t)^2\le c t,\ \ t\ge 0$$
holds for some constant $c>0$. So, by Jensen's inequality,  for any $p\in (0,\aa)$, there exists a constant $c_1>0$ such that
$$  \sup_{x\in M} \E^x [\rho(X^B_{0},X^B_r)^p]=\sup_{x\in M}\E^x[\rho(X_0,X_{S^B_r})^p] \leq c^{p/2} \E \big[(S_r^B)^{\frac{p}{2}}\big]\leq c_1 r^{\frac{p}{2\alpha}},\ \ r\in [0,1],$$
where the last step follows from \eqref{LT} and $B\in \BB_\aa$ from which we find      constants $c_2,c_3>0$ such that for $\vv:=\ff p 2$,
\beg{align*} &\E\big[ (S_r^B)^{\vv }\big]= \ff \vv {\GG(1-\vv)} \int_0^\infty (1-\e^{-r B(t)})t^{-\vv-1}\d t \\
&\le c_2 \int_0^\infty (1-\e^{-c_2 r-c_2 r t^\aa}) t^{-\vv-1}\d t \le c_2 \e^{c_2 r} \int_0^\infty (1-\e^{-c_2 r t^\aa}) t^{-\vv-1}\d t\le c_3 r^{\ff\vv\aa},\ \ r\in [0,1].
\end{align*}
Therefore, there exists a constant $c_4>0$ such that
\beq\label{NN1}\sup_{x\in M} \E^x \big[ \W_p(\mu_t^B,\mu_N^B)\big]\le c_4 (t N^{-1})^{\ff p {2\aa}},\ \ t\ge 1, N\in \mathbb N. \end{equation}

On the other hand,  since $M$ is compact, there exists a constant $c_5>0$ such that
$$\mu(\{\rr(x,\cdot)^p \le r\}) \le c_5r^{\ff d p},\ \ r>0, x\in M.$$
By \cite[Proposition~4.2]{eK}, this implies
$$\W_p(\mu_N^B,\mu)\geq c_6 N^{-\frac{p}{d}},\quad N\in\mathbb{N},t\geq 1$$
for some constant $c_6>0$.
This and \eqref{NN1} yield
\begin{align*}
&\inf_{x\in M} \E^x[\W_p (\mu,\mu_t^B)]  \geq \inf_{x\in M} \E^x[\W_p(\mu,\mu_N^B)] -\sup_{x\in M} \E^x[\W_p(\mu_t^B,\mu_N^B)]\\
&\geq c_6 N^{-\frac{p}{d}} - c_4(tN^{-1})^{\ff p{2\aa}},  \ \ t\geq 1, N\in \mathbb N.
\end{align*}
By taking  $N:=\inf\{n\in\mathbb N: n\ge \dd t^{\ff{d}{d-2\aa}}\}$  for small $\dd>0$, find a constant $c_7>0$ such that for large enough $t>1,$
$$\inf_{x\in M} \E^x[\W_p(\mu,\mu_t^B)]  \ge c_7 t^{\ff p{d-2\aa}}.$$
Hence,  the desired estimate holds.

\subsection{Proof of Theorem \ref{T1.1}(3)}
We only consider the case that    $\alpha=\frac{1}{2},d=3$, since the proof for $\aa=1$ and $d=4$ has been presented in \cite{eWZ}.   In this case, the assertion
  is implied by the following two lemmas   which  essentially due to \cite{eWZ} for $\aa=1$.

\begin{lem}\label{L1}
Let $B(\ll)= \ll^{\ff 1 2} $ and  $d=3$.  If   for any  constant $C>1$  there exist constants $\gg,\varepsilon,  t_0>0$, such that
\begin{equation}\label{pop511}
  \{\E^\nu\W_1(\mu_{t,t^{-\gg}}^B,\mu)\}^2\geq\varepsilon\E^\nu\mu(|\nabla(-L)^{-1}(f_{t,t^{-\gg}}-1)|^2),\quad \nu\in \scr P(C),   t>t_0,
\end{equation}
then the estimate in Theorem $\ref{T1.1}(3)$ holds.
\end{lem}

\begin{proof}  By Lemma \ref{LYP} for $p=1$, it suffices to  prove that for any constant $C>1$,
\begin{equation}\label{pop513}
\liminf_{t\to\infty}t(\log t)^{-1}\inf_{\nu\in \scr P(C)}\{\E^\nu \W_1(\mu_t^B,\mu)\}^2>0.
\end{equation}
By   \eqref{ee:equ01} and \eqref{pop511}, there exists a constant $c_1, t_1>0$ such that
$$\inf_{\nu\in \scr P(C)}\{\E^\nu\W_1(\mu_{t,t^{-\gg}}^B,\mu)\}^2\geq \frac{c_1}{t}\sum_{i=1}^\infty\frac{1}{\lambda_i^{\frac{3}{2}}\e^{2t^{-\gg} \lambda_i}},\ \   t>t_1.$$
Since $d=3$,  \eqref{*2} implies $\ll_i\le c i^{\ff 2 3} $ for some constant $c>0$, so that we find constants $c_2,c_3>0$ such that
$$\inf_{\nu\in \scr P(C)}\{\E^\nu\W_1(\mu_{t,t^{-\gg}}^B,\mu)\}^2\geq\frac{1}{c_2 t} \int_1^\infty \ff {\d s} {s \e^{c_2t^{-\gg} s^{\frac 23}}} \ge \ff{c_3\log t }t,\ \   t>t_1.$$
 Combining this with  \eqref{WW0}, we find a constant $c_4>0$ such that
 $$\inf_{\nu\in \scr P(C)}\{\E^\nu\W_1(\mu_t^B,\mu)\}^2\geq  \ff{ c_4\e^{-2 Kt^{-\gg}} \log t}{t},\ \ \  t>t_1.$$
 This implies  \eqref{pop513}.
\end{proof}

\begin{lem}\label{L2} Let $M=\TO^3, V=0$ and $B(\ll)=\ll^{\ff 1 2}. $ Then for any $\gg\in (0, \ff 2 5)$ there exist    constants $\vv, t_0>0$ such that
 \begin{equation}\label{NB}
  \{\E^\nu\W_1(\mu_{t,t^{-\gg}}^B,\mu)\}^2\geq\varepsilon\E^\nu\mu(|\nabla(-\DD)^{-1}(f_{t,t^{-\gg}}-1)|^2),\quad \nu\in \scr P,   t>t_0.\end{equation}
 \end{lem}

\beg{proof}  The proof is similar to that of \cite[Proposition 5.3]{eWZ} with $X_t^B$ replacing $X_t$.

Let $f_t=(-\Delta)^{-1}(f_{t,t^{-\gamma}}-1)$. It is shown in the proof of \cite[Proposition 5.3]{eWZ} that
$$ \W_1(\mu_{t,t^{-\gamma}}^B,\mu)\geq\beta^{-1} \mu(|\nabla f_t|^2)- K_1\beta^{-3}\mu(|\nabla f_t|^4),\quad \beta>0$$
holds for some   constant $K_1>0$.
If there exist a constant $K_2>0$ such that
\begin{equation}\label{WWN}
\E^\nu\mu(|\nabla f_t|^4)\leq K_2[\E^\nu\mu(|\nabla f_t|^2)]^2,\quad t\geq 2,
\end{equation}
then
$$\E^\nu\W_1(\mu_{t,t^{-\gamma}}^B,\mu)\geq\beta^{-1}\E^\nu\mu(|\nabla f_t|^2)-\beta^{-3}K_1K_2[\E^\nu\mu(|\nabla f_t|^2)]^2,\,\quad \beta>0.$$
Taking $\beta=N\E^\nu[\mu(|\nabla f_t|^2)^{\frac{1}{2}}]$ for large enough $N>1$, we prove \eqref{NB} for some constant $c>0$.
So, it remains to prove \eqref{WWN}.

We identify $\TO$ with $[0,2\pi)$ by the one-to-one map
$$[0,2\pi)\ni s\mapsto \e^{\i s},$$
where i is the imaginary unit. In this way, a point in $\TO^3$ is regarded as a point in $[0,2\pi)^3$, so that $\{\e^{\i\langle m,\cdot\rangle}\}_{m\in\Z^3}$ consist of an eigenbasis of $\Delta$ in the complex $L^2$-space of $\mu$, where $\mu$ is the normalized volume measure on $\TO^3$.  Since $X_t^B$ is generated by $-(-\DD)^{\ff 1 2}$, we have
\beq\label{EG} \E^x \e^{{\rm i} \<m,  X_t^B\>}= \e^{-|m| t}\e^{{\rm i}\<m,x\>},\ \ t\ge 0, x\in \TO^3, m\in \Z^3.\end{equation}
Moreover,
$$f_t:=(-\Delta)^{-1}(f_{t,t^{-\gamma}}-1)=\sum_{m\in\Z^{3}\backslash\{0\}}b_m \e^{-\textnormal{i}\langle m,\cdot\rangle},$$ where
\beq\label{BM}   b_m:=\frac{\e^{-|m|^2 t^{-\gamma}}}{|m|^2t}\int_0^t \e^{\textnormal{i}\langle m,X^B_s\rangle}\,\d s,\ \ m\in \Z^3.\end{equation}
Then
$$|\nabla f_t(x)|^2=-\sum_{m_1,m_2\in \Z^3\backslash\{0\}}\langle m_1,m_2\rangle b_{m_1}b_{m_2}\e^{-\textnormal{i}\langle m_1+m_2,x\rangle},$$
$$|\nabla f_t(x)|^4=\sum_{m_1,\cdots, m_4\in \Z^3\backslash\{0\}}\langle m_1,m_2\rangle\langle m_3,m_4\rangle b_{m_1}b_{m_2}b_{m_3}b_{m_4}\e^{-\textnormal{i}\langle m_1+m_2+m_3+m_4,x\rangle}.$$
Noting that, $\mu(\e^{-\textnormal{i}\langle m,\cdot\rangle})=0$ when $m\neq 0$, we get
\begin{equation}\label{pop532}
\E^\nu\mu(|\nabla f_t|^2)=\sum_{m\in\Z^3\backslash\{0\}}|m|^2\E^{\nu}[b_m b_{-m}],
\end{equation}
\begin{equation}\label{pop533}
\E^\nu\mu(|\nabla f_t|^4)=\sum_{(m_1,m_2,m_3,m_4)\in\Ss}\langle m_1,m_2 \rangle\langle m_3,m_4 \rangle\E^{\nu}[b_{m_1}b_{m_2}b_{m_3}b_{m_4}],
\end{equation}
where ~$\Ss:=\{(m_1,m_2,m_3,m_4)\in\Z^3\backslash\{0\}:m_1+m_2+m_3+m_4=0\}$.~

By \eqref{BM}, we have
$$\E^\nu[b_m b_{-m}]=\frac{\e^{-2|m|^2 t^{-\gamma}}}{|m|^4 t^2}\int_{[0,t]^2}\E^\nu \e^{\textnormal{i}\langle m,X^B_{s_2}-X^B_{s_1}\rangle}\,\d s_1\d s_2.$$
The Markov property and $\eqref{EG}$ yield
\begin{equation}\label{pop534}
\E^\nu(\e^{\textnormal{i}\langle m,X^B_{s_2}-X^B_{s_1}\rangle}|\mathscr{F}_{s_1\wedge s_2})=\e^{-|m||s_1-s_2|},\ \ s_1,s_2\ge 0.
\end{equation}
Then we   find a constant $\kappa>0$  such that
$$\E^\nu[b_m b_{-m}]=\frac{\e^{-2|m|^2 t^{-\gamma}}}{|m|^4 t^2}\int_{[0,t]^2}\e^{-|m||s_1-s_2|}\,\d s_1\d s_2\geq\frac{\kappa \e^{-2|m|^2 t^{-\gamma}}}{|m|^5 t},\quad t\geq 2.$$
Using this and \eqref{pop532}, we get that
\begin{equation}\label{pop535}\begin{split}
&\E^\nu\mu(|\nabla f_t|^2)\geq\sum_{m\in\Z^3\backslash\{0\}}\frac{\kappa \e^{-2|m|^2t^{-\gamma}}}{|m|^3 t}\geq\frac{\kappa_1}{t}\int_1^\infty\frac{\e^{-2s^2 t^{-\gamma}}}{s}\,\d s\\
&\geq\frac{\kappa_1}{t \e^2}\int_1^{t^{\frac{\gamma}{2}}}s^{-1}\,\d s=\frac{\kappa_1\gamma}{2\e^2}(t^{-1}\log t),\quad t\geq 2.
\end{split}\end{equation}

Let $\mathbf{S}$ be the set of all the permutations of $\{1,2,3,4\}$, $D(t)=\{(s_1,s_2,s_3,s_4)\in[0,t]^4:0\le s_1\le s_2\le s_3\le s_4 \le t\}$. We have
\begin{align*}
&\E^\nu[b_{m_1}b_{m_2}b_{m_3}b_{m_4}]\\
&=\frac{\e^{-\sum_{p=1}^4|m_p|^2 t^{-\gamma}}}{t^4\prod_{p=1}^4|m_p|^2}\int_{[0,t]^4}\E^\nu[\e^{\i\langle m_1,X^B_{s_1}\rangle}\e^{\i\langle m_2,X^B_{s_2}\rangle}\e^{\i\langle m_3,X^B_{s_3}\rangle}\e^{\i\langle m_4,X^B_{s_4}\rangle}]\,\d s_1\d s_2\d s_3\d s_4\\
&=\frac{\e^{-\sum_{p=1}^4|m_p|^2 t^{-\gamma}}}{t^4\prod_{p=1}^4|m_p|^2}\sum_{(i,j,k,l)\in\mathbf{S}}\int_{D(t)}\E^\nu[\e^{\i\langle m_i,X^B_{s_1}\rangle}\e^{\i\langle m_j,X^B_{s_2}\rangle}\e^{\i\langle m_k,X^B_{s_3}\rangle}\e^{\i\langle m_l,X^B_{s_4}\rangle}]\,\d s_1\d s_2\d s_3\d s_4
\end{align*}
Since $m_1+m_2+m_3+m_4=0$, by \eqref{EG} and the Markov property we obtain
\begin{equation*}\begin{split}
\E^\nu[\e^{\textnormal{i}\langle m_i,X^B_{s_1}\rangle}\e^{\textnormal{i}\langle m_j,X^B_{s_2}\rangle}\e^{\textnormal{i}\langle m_k,X^B_{s_3}\rangle}\e^{\textnormal{i}\langle m_l,X^B_{s_4}\rangle}]=\e^{-|m_l|(s_4-s_3)-|m_l+m_k|(s_3-s_2)-|m_i|(s_2-s_1)}.
\end{split}\end{equation*}
Thus,
\begin{equation}\label{pop536}\begin{split}
&\frac{t^4\prod_{p=1}^4|m_p|^2}{\e^{-\sum_{p=1}^4|m_p|^2 t^{-\gamma}}}\E^\nu[b_{m_1}b_{m_2}b_{m_3}b_{m_4}]\\
&=\sum_{(i,j,k,l)\in\mathbf{S}}\int_{D(t)}\e^{-|m_l|(s_4-s_3)-|m_l+m_k|(s_3-s_2)-|m_i|(s_2-s_1)}\,\d s_1\d s_2\d s_3\d s_4.
\end{split}\end{equation}
If $m_l+m_k=0$, then
\begin{equation*}\begin{split}
&\int_{D(t)}\e^{-|m_l|(s_4-s_3)-|m_l+m_k|(s_3-s_2)-|m_i|(s_2-s_1)}\,\d s_1\d s_2\d s_3\d s_4\\
&=\int_0^t\int_{s_1}^t\int_{s_2}^t\int_{s_3}^t \e^{-|m_l|(s_4-s_3)}\e^{-|m_i|(s_2-s_1)}\,\d s_4\d s_3\d s_2\d s_1\leq\frac{t^2}{|m_i||m_l|}.
\end{split}\end{equation*}
If $m_l+m_k\neq0$, then
\begin{equation*}\begin{split}
&\int_{D(t)}\e^{-|m_l|(s_4-s_3)-|m_l+m_k|(s_3-s_2)-|m_i|(s_2-s_1)}\,\d s_1\d s_2\d s_3\d s_4\\
&=\int_0^t\int_{s_1}^t\int_{s_2}^t\int_{s_3}^t \e^{-|m_l|(s_4-s_3)}\e^{-|m_l+m_k|(s_3-s_2)}\e^{-|m_i|(s_2-s_1)}\,\d s_4\d s_3\d s_2\d s_1\\
&\leq\frac{t}{|m_i||m_l+m_k||m_l|}.
\end{split}\end{equation*}
Combining these with  \eqref{pop536} leads to
$$\E^\nu[b_{m_1}b_{m_2}b_{m_3}b_{m_4}]\leq \frac{\e^{-\sum_{p=1}^4|m_p|^2 t^{-\gamma}}}{\prod_{p=1}^4|m_p|^2}\sum_{(i,j,k,l)\in\mathbf{S}}\left\{\frac{t^{-2}1_{\{m_l+m_k=0\}}}{|m_i||m_l|}+\frac{t^{-3}1_{\{m_l+m_k\neq 0\}}}{|m_i||m_l+m_k||m_l|}\right\}.$$
Therefore,   by \eqref{pop533}, we find  a constant $c>0$  such taht
\begin{equation}\label{pop537}
\E^\nu\mu(|\nabla f_t|^4)\leq c(I_1+I_2),\quad t\geq 2,
\end{equation}
holds for
$$I_1:=\frac{1}{t^2}\sum_{a,b\in\Z^3\backslash\{0\}}\frac{1}{|a|^3|b|^3}\e^{-2(|a|^2+|b|^2)t^{-\gamma}},$$
$$I_2:=\frac{1}{t^3}\sum_{\substack{m_1,m_2,m_3,m_4\in\Z^3\backslash\{0\}\\m_3+m_4\neq 0 }}\frac{\e^{-\sum_{p=1}^4|m_p|^2 t^{-\gamma}}}{|m_1|^2|m_2||m_3||m_3+m_4||m_4|^2}.$$
It is easy to see that there exists constants $c_1,c_2>0$, such that
\begin{equation}\label{pop538}
I_1\leq\frac{c_1}{t^2}\left(\int_1^\infty\frac{\e^{-2s^2 t^{-\gamma}}}{s}\,\d s\right)^2\leq c_2(t^{-1}\log t)^2,\quad  t\geq 2,
\end{equation}
and similarly
$$\sum_{m\in\Z^3\backslash\{0\}}\frac{\e^{-|m|^2t^{-\gamma}}}{|m|^2}\leq c_2t^{\frac{\gamma}{2}},\quad \sum_{m\in\Z^3\backslash\{0\}}\frac{\e^{-|m|^2t^{-\gamma}}}{|m|}\leq c_2t^{\gamma},\quad t\geq 2,$$
Then by reformulating $I_2$ as
$$I_2=\frac{1}{t^3}\left(\sum_{m_1\in\Z^3\backslash\{0\}}\frac{\e^{-|m_1|^2t^{-\gamma}}}{|m_1|^2}\right)\left(\sum_{m_2\in\Z^3\backslash\{0\}}\frac{\e^{-|m_2|^2 t^{-\gamma}}}{|m_2|}\right)
\sum_{\substack{m_3,m_4\in\Z^3\backslash\{0\}\\m_3+m_4\neq 0 }}\frac{\e^{-(|m_3|^2 +|m_4|^2)t^{-\gamma}}}{ |m_3||m_3+m_4||m_4|^2},$$
we find a constant $c_3>0$ such that
\begin{equation}\label{pop539}
I_2 \leq c_3^2 t^{\frac{3\gamma}{2}-3}\sum_{m_4\in\Z^3\backslash\{0\}}\frac{\e^{-|m_4|^2t^{-\gamma}}}{|m_4|^2}
\sum_{m_3\in\Z^3\backslash\{0,-m_4\}}\frac{\e^{-|m_3|^2 t^{-\gamma}}}{|m_3||m_3+m_4|}.
 \end{equation}
Write
\begin{equation}\label{pop540}\begin{split}
\sum_{m_3\in\Z^3\backslash\{0,-m_4\}}\frac{\e^{-|m_3|^2 t^{-\gamma}}}{|m_3||m_3+m_4|}=:J_1+J_2+J_3
\end{split}\end{equation}
for
$$J_1:=\sum_{\substack{m_3\in\Z^3\backslash\{0,-m_4\}\\|m_3|\leq\frac{|m_4|}{2}}}\frac{\e^{-|m_3|^2 t^{-\gamma}}}{|m_3||m_3+m_4|},$$
$$J_2:=\sum_{\substack{m_3\in\Z^3\backslash\{0,-m_4\}\\\frac{|m_4|}{2}<|m_3|\leq 2|m_4|}}\frac{\e^{-|m_3|^2 t^{-\gamma}}}{|m_3||m_3+m_4|},$$
$$J_3:=\sum_{\substack{m_3\in\Z^3\backslash\{0,-m_4\}\\|m_3|> 2|m_4|}}\frac{\e^{-|m_3|^2 t^{-\gamma}}}{|m_3||m_3+m_4|}.$$
On the region $\{m_3\in\Z^3\backslash\{0,-m_4\}:|m_3|\leq\frac{|m_4|}{2}\}$  we   find a constant $c_4>0$ such that
\begin{equation}\label{pop541}
J_1\leq\frac{2}{|m_4|}\sum_{m_3\in\Z^3\backslash\{0\}}\frac{\e^{-|m_3|^2t^{-\gamma}}}{|m_3|}\leq\frac{c_4 t^\gamma}{|m_4|},\quad t\geq2.
\end{equation}
Next, on the region $\{m_3\in\Z^3\backslash\{0,-m_4\}:|m_3|>2|m_4|\}$, we have $|m_3+m_4|\sim |m_3|$ and $|m_3|^2\geq\frac{|m_3|^2}{2}+2|m_4|^2$, so we find a  a constant $c_5>0$ such that
$$J_3\leq 4\sum_{\substack{m_3\in\Z^3\backslash\{0\}\\|m_3|> 2|m_4|}}\frac{\e^{-|m_3|^2t^{-\gamma}}}{|m_3|^2}\leq 4 \e^{-2|m_4|^2t^{-\gamma}}\sum_{m_3\in\Z^3\backslash\{0\}}\frac{\e^{-\frac{|m_3|^2t^{-\gamma}}{2}}}{|m_3|^2}\leq c_5 t^{\frac{\gamma}{2}}\e^{-2|m_4|^2t^{-\gamma}}.$$
This together with  $\e^{-s}\leq s^{-\ff 1 2}$ gives
\begin{equation}\label{pop542}
J_3\leq\frac{c_5t^{\gamma}}{|m_4|},\quad t\geq 2.
\end{equation}
Finally, on the region $\{m_3\in\Z^3\backslash\{0,-m_4\}:\frac{|m_4|}{2}<|m_3|\leq2|m_4|\}$, we have $|m_3|\sim|m_4|$ and $1\leq|m_3+m_4|\leq 3|m_4|$, so that  there for a constant $c_6>0$
$$J_2\leq\frac{2\e^{-\frac{|m_4|^2t^{-\gamma}}{4}}}{|m_4|}\sum_{1\leq|m_3+m_4|\leq 3|m_4|}\frac{1}{|m_3+m_4|}\leq c_6|m_4| \e^{-\frac{|m_4|^2 t^{-\gamma}}{4}}.$$
By $\e^{-s}\leq s^{-1}$, we get the upper estimate of $J_2$,
$$J_2\leq\frac{c_7 t^{\gamma}}{|m_4|},\quad t\geq 2.$$
Combining this with \eqref{pop539},\eqref{pop540},\eqref{pop541} and \eqref{pop542}, we find a constant $c_8>0$ such that
$$I_2\leq c_8t^{\ff 5 2 \gamma-3}\log t,\quad t\geq 2.$$
Substituting this and \eqref{pop538} into \eqref{pop537}, and combining with \eqref{pop535}, we prove \eqref{WWN}. The proof is finished.

\end{proof}

\section{Proof of Theorem \ref{T1.2}}

(1)  By Lemma \ref{LYP} for $p=2$,   it suffices to prove
 \beq\label{UPS'}  \limsup_{t\to\infty}\sup_{\nu\in \scr P(C)}\left\{t\E^\nu[\W_2(\mu_t^B,\mu)^2]\right\}\leq \sum_{i=1}^\infty\frac{2}{\lambda_iB(\ll_i)},\ \ C>1.\end{equation}
By the triangle inequality of $\W_2$ and Lemma \ref{lemma31},  for any $\varepsilon>0$ there exists a constant $c(\vv)>0$ such that
\begin{align*}
&\E^\nu[\W_2(\mu_t^B,\mu)^2]\\ &\leq (1+\varepsilon)\E^\nu[\W_2(\mu_{t,r_t,r_t}^B,\mu)^2]+2(1+\varepsilon^{-1})\big\{\E^\nu[\W_2(\mu_{t,r_t}^B,\mu_{t,r_t,r_t}^B)^2]+\E^\nu[\W_2(\mu_t^B,\mu_{t,r_t}^B)^2]\big\}\\
&\leq (1+\varepsilon)\E^\nu[\W_2(\mu_{t,r_t,r_t}^B,\mu)^2]+c(\vv) r_t,
\end{align*}
where $r_t=t^{-\beta},\beta\in(1,\frac{2q}{d}),q\in(\frac{d}{2}\vee 1,\frac{d}{d-2\alpha}),t\geq 1$.
Since   $\frac{d\mu_{t,r_t,r_t}}{d\mu}=(1-r_t)f_{t,r_t}+r_t$, by   combining this with  Lemma \ref{ee:lem1} and  H\"{o}lder's inequality, we obtain
\begin{align*}
&\E^\nu[\W_2(\mu_{t,r_t,r_t}^B,\mu)^2]\leq\E^\nu\int_M\frac{|\nabla L^{-1}(f_{t,r_t}-1)|^2}{\mathscr{M}((1-r_t)f_{t,r_t}+r_t,1)}\,d\mu\\
&\leq\E^\nu\int_M\left\{|\nabla L^{-1}(f_{t,r_t}-1)|^2+|\nabla L^{-1}(f_{t,r_t}-1)|^2|\mathscr{M}((1-r_t)f_{t,r_t}+r_t,1)^{-1}-1|\right\}\,d\mu\\
&\leq\E^\nu\int_M|\nabla L^{-1}(f_{t,r_t}-1)|^2\,d\mu+\left(\E^\nu\int_M|\nabla L^{-1}(f_{t,r_t}-1)|^{2p}\,d\mu\right)^{\frac{1}{p}}\\
& \times\left(\E^\nu\int_M|\mathscr{M}((1-r_t)f_{t,r_t}+r_t,1)^{-1}-1|^{\frac{p}{p-1}}\right)^{\frac{p-1}{p}}.
\end{align*}
Since $B\in \BB^\aa$,  by Lemma \ref{ee:lem2}, Lemma \ref{lem35} and \eqref{lem331}, this implies
$$\limsup_{t\to\infty}\sup_{\nu\in\scr{P} (C)}\left\{t\E^\nu[\W_2(\mu_{t,r_t,r_t}^B,\mu)^2]\right\}\leq\sum_{i=1}^\infty\frac{2}{\ll_i B(\ll_i)}.$$
Combining this with  Lemma \ref{lemma31} for $\vv=r=r_t:=t^{-\bb}$ where $\bb>1$, we prove     \eqref{UPS'}.

 (2) By Lemma \ref{LYP} for $p=2$, it suffices to prove
\beq\label{UPS2}  \limsup_{t\to\infty}\sup_{\nu\in \scr P(C)}\left\{t^{\ff 2{d-2\aa}}\E^\nu[\W_2(\mu_t^B,\mu)^2]\right\}<\infty,\ \ C>1.\end{equation}
  Let $r: (1,\infty)\to (0,1)$ to be determined.  By \cite{L17}, we have
  $$t\W_2(\mu_{t,r}^B,\mu)^2\le 4  \Xi_r(t),\ \ t,r>0.$$
  Combining this with  Lemma~\ref{lemma31} and Lemma~\ref{ee:lem2}, we find a constant $c_0>0$ such that
\begin{equation}\label{ee:sup1}\begin{split}
\E^\nu[\W_2(\mu_t^B,\mu)^2]&\leq 2\E^\nu[\W_2(\mu_t^B,\mu_{t,r_t}^B)^2]+2\E^\nu[\W_2(\mu_{t,r_t}^B,\mu)^2]\\
&\leq c_0 r_t+c_0\frac{\|h_\nu\|_\infty}{t}\sum_{i=1}^\infty\frac{1}{\lambda_i^{1+\alpha}\e^{2r_t\lambda_i}},\quad t>1.
\end{split}\end{equation}
By \eqref{*2}, there exists constants $c_2,c_3>0$ such that
$$\sum_{i=1}^\infty\frac{1}{\lambda_i^{1+\alpha}\e^{2r_t\lambda_i}}\leq    c_2\int_1^\infty s^{-\frac{2(1+\alpha)}{d}}\e^{-c_3 r_t s^{\frac{2}{d}}}\,\d s,$$
so that \eqref{ee:sup1} implies
\beq\label{NNN} \sup_{\nu\in\scr{P}(C)} \E^\nu[\W_2(\mu_t^B,\mu)^2]\le c r_t + \ff c t \int_1^\infty s^{-\frac{2(1+\alpha)}{d}}\e^{-c_3 r_t s^{\frac{2}{d}}}\,\d s,\ \ t> 1, r_t>0\end{equation}
for some constant $c>0$ depending on $C$.

Since $d>2(1+\alpha)$, we find a constant $c_4>0$ such that
$$\int_1^\infty s^{-\frac{2(1+\alpha)}{d}}\e^{-c_3 r_t s^{\frac{2}{d}}}\,\d s=\int_{r_t^\frac{d}{2}}^\infty (r_t^{-\frac{d}{2}}u)^{-\frac{2(1+\alpha)}{d}}\e^{-c_3 u^{\frac{2}{d}}}r_t^{-\frac{d}{2}}\,\d u
\leq c_4 r_t^{-\frac{d-2(1+\alpha)}{2}},\ \ t>1.$$
Combining this with \eqref{NNN} and taking
$$r_t= t^{-\frac{2}{d-2\alpha}},\ \ t>1,$$
we prove \eqref{UPS2}.

(3)  Since $d=2(1+\aa)$,  for any $c>0$ there exists a constant $c_1>0$ such that  there exist a constants $c_1>0$ such that
$$
 \int_1^\infty s^{-\frac{2(1+\alpha)}{d}}\e^{-cr_t s^{\frac{2}{d}}}\,\d s   \leq c_1 \ln(1+r_t^{-1}),\ \ t>1,
$$
so that \eqref{NNN} implies
$$  \E^\mu [\W_2(\mu_t^B,\mu)^2] \le c' r_t+ c' t^{-1} \log (1+r_t^{-1}),\ \ t>1$$
for some constant $c'>0$.
Taking $r_t= t^{-1}\log (1+ t^{-1})$ for $t\ge 2$, we find a constant $c_2>0$ such that
$$  \E^\mu [\W_2(\mu_t^B,\mu)^2] \le c_2 t^{-1}\log (1+t),\  \ t\ge 2.$$
Since $\E^\nu\le \|h_\nu\|_\infty \E^\mu$ for $\nu=h_\nu\mu$, combining this with   Lemma \ref{LYP} for $p=2$ and $\vv=1$, we obtain
$$\limsup_{t\to\infty} \ff{t}{\log t} \sup_{x\in M} \E^x [\W_2(\mu_t^B,\mu)^2] <\infty.$$

\section{Proof of Theorem \ref{T1.4}}
\begin{proof}
By \eqref{N*N}, \eqref{B*B} and noting that  the Markov property implies
$$\P^x(t\W_2(\tt\mu_{\vv,t}^B, \mu)^2<a)= \P^{\nu_{x,\vv}}(t\W_2(\mu_t^B,\mu)^2<a),\ \ a\ge 0,$$
it suffices to prove that for any $C>1$,
\beq\label{XW1} \liminf_{t\to\infty}\inf_{\nu\in \scr P(C)} \P^\nu(t\W_2(\mu_t^B,\mu)^2<a) \ge F(a),\ \ a\ge 0,\end{equation}
\beq\label{XW2} \limsup_{t\to\infty}\sup_{\nu\in \scr P(C)} \P^\nu(t\W_2(\mu_t^B,\mu)^2<a) \le F(a),\ \ a\ge 0.\end{equation}
It is easy to see that \eqref{XW2} follows from  Theorem \ref{T2.1}(2) and \eqref{WW0}.

To prove   \eqref{XW1}, let $\gg>1$ be in Lemma \ref{lem32}(2), and denote
\beg{align*}& \tt\Xi(t):= t \int_M \ff{|\nn L^{-1}(f_{t,t^{-\gg}}-1)|^2}{\scr M((1-t^{-\gg})f_{t,t^{-\gg}}+t^{-\gg}, 1)}\d\mu,\\
&\Xi(t):= \Xi_{t^{-\gg}}(t)= t\mu\big( |\nn L^{-1}(f_{t,t^{-\gg}}-1)|^2\big),\ \ t>1.\end{align*}
 Then Lemma \ref{lem32}(2) and Lemma \ref{lem35} yield
$$   \limsup_{t\to\infty}\sup_{\nu\in \scr P(C)} \P^\nu(|\tt \Xi(t)-\Xi(t)|>\vv)=0,\ \ \vv>0.$$
Combining this with Lemma \ref{ee:lem1}, \eqref{ee:equ01} and noting that $\sum_{i=1}^\infty\ff 1 {\ll_iB(\ll_i)}<\infty,$ we prove \eqref{XW1}.
\end{proof}

\end{document}